\documentclass{amsart}

\usepackage{amsmath}
\usepackage{amssymb}
\usepackage{amsthm}
\usepackage{esint}
\usepackage[inline]{enumitem}
\usepackage{color}
\usepackage{orcidlink}

\theoremstyle{plain}
\newtheorem{theorem}{Theorem}[section]
\newtheorem{corollary}[theorem]{Corollary}
\newtheorem{lemma}[theorem]{Lemma}
\newtheorem{proposition}[theorem]{Proposition}

\theoremstyle{definition}
\newtheorem{definition}[theorem]{Definition}
\newtheorem{remark}[theorem]{Remark}

\theoremstyle{remark}
\newtheorem{example}[theorem]{Example}

\numberwithin{theorem}{section}
\numberwithin{equation}{section}

\newcommand{\Z}{\mathbb{Z}}

\newcommand{\R}{\mathbb{R}}

\newcommand{\dist}{\mathrm{dist}}
\newcommand{\diam}{\mathrm{diam}}
\newcommand{\inrad}{\mathrm{inrad}}

\newcommand{\cl}{\overline}

\newcommand{\loc}{\mathrm{loc}}

\DeclareMathOperator*{\divergence}{div}

\newcommand{\laplacian}{\Delta}

\DeclareMathOperator*{\osc}{osc}
\DeclareMathOperator*{\esssup}{ess\,sup}
\DeclareMathOperator*{\essinf}{ess\,inf}

\DeclareMathOperator*{\spt}{supp}
\newcommand{\capacity}{\mathrm{cap}}

\newcommand{\wkto}{\rightharpoonup}

\newcommand{\trinorm}[1]
{{
    \left\vert\kern-0.20ex\left\vert\kern-0.20ex\left\vert
    #1 
    \right\vert\kern-0.20ex\right\vert\kern-0.20ex\right\vert
}}

\newcommand{\M}{\mathcal{M}}

\begin{document}


\title[Global H\"{o}lder solvability]{Global H\"{o}lder solvability of linear and quasilinear Poisson equations}
\author{Takanobu Hara \, \orcidlink{0000-0001-6510-0422}}
\email{takanobu.hara.math@gmail.com}
\address{
Mathematical Institute, Tohoku University,
6-3, Aramaki Aza-Aoba, Aoba-ku, Sendai, Miyagi 980-8578, Japan}
\date{\today}
\subjclass[2020]{35J92; 35J25; 31C45} 
\keywords{potential theory, $p$-Laplacian, quasilinear elliptic equation, boundary value problem, boundary regularity}



\begin{abstract}
We establish an existence result for globally continuous weak solutions to elliptic equations of the $p$-Poisson type.
This result significantly improves Theorem 8.30 in Gilbarg-Trudinger (1983) and offers a novel contribution for the classical Poisson equation on Lipschitz domains,
ensuring global H\"{o}lder continuity of solutions under a minimal assumption on the right-hand side.
Applications of this result to embedding theorems are also discussed.
\end{abstract}




\maketitle


\section{Introduction}\label{sec:introduction}

\subsection{Main results}

The purpose of this paper is to establish an existence result for globally continuous weak solutions to
elliptic equations of the $p$-Poisson type.
This result significantly improves \cite[Theorem 8.30]{MR1814364}
and is also new for the classical Poisson equation, ensuring the existence of
H\"{o}lder continuous solutions under a minimal assumption on the right-hand side. 

We now state our main result. 
Throughout the paper, we consider the elliptic equation
\begin{equation}\label{eqn:DE}
\begin{cases}
- \divergence \mathcal{A}(x, \nabla u) = \nu & \text{in} \ \Omega, \\
u = 0 & \text{on} \ \partial \Omega,
\end{cases}
\end{equation}
where 
$\Omega$ is a bounded open set in $\R^{n}$ ($n \ge 2$),
$\divergence \mathcal{A}(x, \nabla \cdot)$ is a $p$-Laplacian ($1 < p \le n$) type elliptic operator,
and $\nu$ is a locally finite signed Radon measure on $\Omega$.
The precise assumptions on
$\mathcal{A} \colon \Omega \times \R^{n} \to \R^{n}$ are as follows:
For each $z \in \R^{n}$, $\mathcal{A}(\cdot, z)$ is measurable, for each $x \in \Omega$, $\mathcal{A}(x, \cdot)$ is continuous,
and there exists $1 \le L < \infty$ such that
\begin{align}
\mathcal{A}(x, z) \cdot z \ge |z|^{p}, \label{eqn:coercive}
\\ 
|\mathcal{A}(x, z)| \le L  |z|^{p - 1},  \label{eqn:growth}
\\
\left( \mathcal{A}(x, z_{1}) - \mathcal{A}(x, z_{2}) \right) \cdot (z_{1} - z_{2}) > 0, \label{eqn:monotonicity}
\end{align}
for all $x \in \Omega$, $z, z_{1}, z_{2} \in \R^{n}$, $z_{1} \not = z_{2}$ and $t \in \R$.
For an open set $U$ and a compact subset $K \subset U$,
we define the variational $p$-capacity $\capacity_{p}(K, U)$ by
\begin{equation}\label{eqn:variational_capacity}
\capacity_{p}(K, U)
:=
\inf
\left\{
\int_{\R^{n}} |\nabla u|^{p} \, dx \colon u \in C_{c}^{\infty}(U), \ u \ge 1 \, \text{on} \, K
\right\}.
\end{equation}
Our main result is as follows.
\begin{theorem}\label{thm:main}
Assume that $\Omega$ satisfies the capacity density condition
\begin{equation}\label{eqn:CDC_intro}
\exists \gamma > 0 \quad \text{s.t.} \quad
\frac{ \capacity_{p}( \cl{B(\xi, R)} \setminus \Omega, B(\xi, 2R)) }{ \capacity_{p}( \cl{B(\xi, R)}, B(\xi, 2R)) } \ge \gamma
\quad \forall R > 0, \ \forall \xi \in \partial \Omega.
\end{equation}
Let $\nu$ be a locally finite signed Radon measure satisfying
\begin{equation}\label{eqn:RZ}
|\nu|(B(x, r))
\le
M^{p - 1}
r^{\lambda}
\quad
\forall x \in \Omega, \, 0 < \forall r < \delta(x) / 2,
\end{equation}
where $M$ is a constant, $n - p <  \lambda \le n$, and $\delta(x) := \dist(x, \partial \Omega)$.
Then, there exists a unique weak solution
$u \in W^{1, p}_{\loc}(\Omega) \cap C(\cl{\Omega})$ to \eqref{eqn:DE}.
Moreover, $u \in C^{\beta_{1}}( \cl{\Omega} )$ and satisfies
\begin{equation}\label{eqn:hoelder_esti}
\sup_{ \substack{x, y \in \Omega \\ x \neq y} } \frac{|u(x) - u(y)|}{|x - y|^{\beta_{1}}}
\le
C \, \inrad(\Omega)^{\beta - \beta_{1}} M,
\end{equation}
where $C$ and $0 < \beta_{1} \le \min\{ \beta, 1 \}$ are positive constants depending only on $n$, $p$, $L$, $\gamma$ and $\lambda$,
$\inrad(\Omega) := \sup_{x \in \Omega} \delta(x)$, and $\beta = (\lambda - n + p) / (p - 1)$.
\end{theorem}

We give simple examples where Theorem \ref{thm:main} applies.
When $\mathcal{A}(x, z) = |z|^{p - 2} z$, the operator $\divergence \mathcal{A}(x, \nabla \cdot)$ is called the $p$-Laplacian and denoted by $\laplacian_{p}$. 
If $\mathcal{A}(x, z) = z$, this reduces to the classical Laplacian $\laplacian$.
The condition \eqref{eqn:CDC_intro} is often used in the study of Hardy's inequality (e.g., \cite{MR856511, MR946438, MR1010807, MR4306765})
and boundary regularity estimates (e.g., \cite[Chapter 6]{MR2305115}).
A straightforward way to verify \eqref{eqn:CDC_intro} is to use exterior cones.
In particular, the Poisson equation on a Lipschitz domain provides an application of Theorem \ref{thm:main}.

Theorem \ref{thm:main} yields a compact embedding theorem as a corollary.
This can be interpreted as a global version of \cite[Sect. 11.9.1, Theorem 3]{MR2777530} under \eqref{eqn:CDC_intro}.
Details are provided in Section \ref{sec:embedding}.

\subsection{Background and contribution of the paper}

The generalized Poisson equation \eqref{eqn:DE} remains an interesting topic despite substantial research into various aspects,
including existence of solutions, their regularity, and their application to embedding theorems.
It is well-known that  there exists a finite energy solution $u \in W_{0}^{1, p}(\Omega)$ if and only if $\nu \in (W_{0}^{1, p}(\Omega))^{*}$
(see, e.g.,  \cite{MR0259693}).
Regularity results for them can be found in famous monographs \cite{MR244627, MR1814364, MR1461542, MR2305115}.
Theory of generalized solutions to finite signed measure data was developed in \cite{MR1025884,MR1205885,MR1760541}.
Definitions of solutions to infinite measure data are given in \cite{MR1205885, MR1655522, MR1955596},
but their existence problem seems to be still open (\cite[Problem 2]{MR1955596}).
See \cite{MR4309811, MR4493271} for partial answers.

The two existence results above are not suitable for considering $W^{1, p}_{\loc}(\Omega) \cap C(\cl{\Omega})$-solutions.
This class is a natural generalization of the class of classical solutions $C^{2}(\Omega) \cap C(\cl{\Omega})$.
The existence of solutions in this class is also known to lead to a certain type of embedding theorem
(e.g., \cite{MR1618334, MR856511, MR1734322}), which are of independent interest.
Unfortunately, there is no inclusion between this class and $W_{0}^{1, p}(\Omega)$.
According to \cite{MR1734322, MR1747901, MR3881877}, let us examine the generalized energy integral 
\[
\int_{\Omega} u^{\gamma} \, d \nu,
\]
where $0 < \gamma < \infty$ and $u$ is a solution to \eqref{eqn:DE}.
The cases $\gamma= 1$ and $\gamma= 0$ correspond to the energy of $u$ and of the mass of $\nu$, respectively.
Our case $\gamma = \infty$ is the most interesting from the point of view of boundary behavior, and methods for other $\gamma$ are clearly inadequate.

The use of supersolutions and the comparison principle are a standard method to obtain globally continuous solutions.
However, verifying the existence of supersolutions in our setting is itself a challenge.
By performing direct calculations with respect to $\delta$, various boundary regularity results can be derived.
This method yields good results when $\mathcal{A}$ and $\Omega$ have a specific structure
(see, e.g., \cite{MR603385, MR700735, MR2193183, MR2286361, MR2286038}).
In contrast, it is well established that the existence of merely continuous solutions requires only
very mild assumptions on $\mathcal{A}$ and $\Omega$.
A distinct gap persists between these results, and it remains unfilled.

In Theorem \ref{thm:barrier}, we solve this problem using combination of a construction method of barriers in \cite{MR856511}
and a boundary regularity estimate Lemma \ref{lem:boundary_hoelder_esti}.
This systematic construction method of supersolutions constitutes the essential contribution of this paper.
This work extends the results from \cite{hara2022strong}.

\subsection{Compare with known results}

In contrast to \cite{MR856511} and \cite{hara2022strong}, we establish the H\"{o}lder norm estimate \eqref{eqn:hoelder_esti} in this paper.
Such a statement seems new even for the Poisson equation on smooth domains since global properties of  \eqref{eqn:RZ} has not yet been discussed.
By \cite[Theorem 3.2]{MR998128},
if there exists a weak supersolution $u \in W^{1, p}_{\loc}(\Omega) \cap C^{\beta}(\cl{\Omega})$ to $- \divergence \mathcal{A}(x, \nabla u) = \nu \ge 0$ in $\Omega$,
then $\nu$ satisfies \eqref{eqn:RZ} with 
\[
M = C \sup_{ \substack{x, y \in \Omega \\ x \neq y} } \frac{|u(x) - u(y)|}{|x - y|^{\beta}}
\quad \text{and} \quad
\lambda = n - p + \beta(p - 1),
\]
where $C$ is a constant independent of $u$ and $\nu$. 
Therefore, our result provides a criterion for the existence of H\"{o}lder continuous solutions (Corollary \ref{cor:criterion}).

The condition \eqref{eqn:RZ} is commonly used in interior H\"{o}lder estimates
(e.g., \cite{MR120446, MR0271383, MR644024, MR0964029, MR998128, MR1283326, MR1290667, MR1354887, MR1887015}),
while its counterpart in boundary estimates is believed -- without any justification -- to be 
\begin{equation}\label{eqn:lem2.7}
|\nu|(B(x, r) \cap \Omega)
\le
M^{p - 1}
r^{\lambda}
\quad
\forall x \in \Omega, \, 0 < \forall r < \infty.
\end{equation}
(See Lemma \ref{lem:boundary_hoelder_esti} for details.)
It is clear that \eqref{eqn:lem2.7} is stronger than \eqref{eqn:RZ}, and
ironically, the most notable difference arises when $\nu = \delta^{-t} m$, where $m$ is the Lebesgue measure.
The condition \eqref{eqn:RZ} is fulfilled if $t < p$ (see Corollary \ref{cor:existence_for_functions}), whereas \eqref{eqn:lem2.7} requires $t < p / n$.
That is, our result properly accounts for the essentially one-dimensional structure of boundary value problems that has previously been neglected.
Lemma \ref{lem:boundary_hoelder_esti} also has the drawback that it assumes the finiteness of the energy of $u$.
The new discussion in this paper can be interpreted as removing these assumptions by applying Lemma \ref{lem:boundary_hoelder_esti} infinitely many times.
Example \ref{exa:sequence_of_finite_energy_sols} below shows that neither \eqref{eqn:lem2.7} nor finiteness of energy is necessary for H\"{o}lder continuity of solutions.

\begin{example}
Let $n \ge 2$, $p = 2$ and $0 < \beta < 1 / 2$.
Let $u(x) = (1 - |x|^{2})^{\beta}$.
Then,
\[
\begin{cases}
- \laplacian u = 2 \beta (1 - |x|^{2})^{\beta - 2} \{ n (1 - |x|^{2}) + 2 (1 - \beta) |x|^{2} \} & \text{in} \ B(0, 1),
\\
u = 0 & \text{on} \ \partial B(0, 1).
\end{cases}
\]
In this case, $u \in C^{\beta}(\cl{B(0, 1)}) \setminus H^{1}_{0}(B(0, 1))$.
Hence, we cannot apply known regularity estimate.
\end{example}

\begin{example}\label{exa:sequence_of_finite_energy_sols}
Let $n \ge 3$, $p = 2$ and $0 < \beta < 1 / 2$.
For $0 < R < 1$, define $u_{R}(x) = u_{R}(|x|)$ on $B(0, 1)$ as follows:
\[
u_{R}(x)
:=
\begin{cases}
\displaystyle
\frac{(1 - R)^{\beta}}{ R^{2 - n} - 1 } \left( |x|^{2 - n} - 1\right) & \text{if} \  R < |x| < 1,
\\
(1 - R)^{\alpha} & \text{otherwise.}
\end{cases}
\]
Direct computation shows that
\[
\| u_{R} \|_{C^{\beta}(\cl{ B(0, 1) })} = 1
\]
and
\[
\| \nabla u_{R} \|_{L^{2}(B(0, 1))}^{2}
=
C(n) \frac{ (1 - R)^{2 \beta} }{ R^{2 - n} - 1 }.
\]
The functions $\{ u_{R} \}$ satisfy a uniform H\"{o}lder estimate and converge uniformly to zero as $R \to 1$. 
However, their Dirichlet energies tend to $+ \infty$.
\end{example}

As will be shown later in Section \ref{sec:examples}, Theorem \ref{thm:main} also applies to these examples.
In fact, it yields a uniform $C^{\beta_{1}}(\cl{\Omega})$ estimate for some $0 < \beta_{1} \le \beta$ and correctly explains the compactness of $\{ u_{R} \}$.
More precisely, we have the explicit identity $- \laplacian u_{R} = f(R) \sigma_{R}$, 
where $\sigma_{R}$ denotes the surface measure on $\partial B(0, R)$ and $f(R)$ is controlled by $(1 - R)^{\beta - 1}$.
Furthermore, by Corollary \ref{cor:surface}, $\nu_{R}$ satisfies \eqref{eqn:RZ} for $\lambda = n - 2 + \beta$, where the constant $M$ is independent of $R$.


\subsection*{Organization of the paper}
Section \ref{sec:preliminaries} provides some auxiliary results.
In Section \ref{sec:barrier}, we establish a constructive method of supersolutions (Theorem \ref{thm:barrier}).
In Section \ref{sec:existence}, we prove our main theorem (Theorem \ref{thm:main}) using results in Section \ref{sec:barrier}.
In Section \ref{sec:embedding}, we prove a compact embedding theorem (Theorem \ref{thm:compact_embedding}) using Theorem \ref{thm:main}.
Examples of measures satisfying our assumptions are discussed in Section \ref{sec:examples}.

\subsection*{Notation}
Throughout the paper, $\Omega$ is an open set in $\R^{n}$, and $\Gamma$ is a closed subset of $\partial \Omega$.
The distance from $\Gamma$ is denoted by $\delta_{\Gamma}(\cdot)$.
If $\Gamma = \partial \Omega$, we write $\delta_{\partial \Omega}(\cdot)$ as $\delta(\cdot)$ simply.
We define the \textit{inradius} of $\Omega$ by $\inrad(\Omega) := \sup_{x \in \Omega} \delta(x)$.
The set of all nonnegative Radon measures on $\Omega$ is denoted by $\M^{+}(\Omega)$.
\begin{itemize}
\item
$\mathbf{1}_{E}(x) :=$ the indicator function of a set $E$.
\item
$C_{c}(\Omega) :=$
the set of all continuous functions with compact support in $\Omega$.
\item
$C_{c}^{\infty}(\Omega) := C_{c}(\Omega) \cap C^{\infty}(\Omega)$.
\item
$L^{p}(\Omega; \mu) :=$ the $L^{p}$ space with respect to $\mu \in \M^{+}(\Omega)$.
\end{itemize}
When the Lebesgue measure must be indicate clearly, we use the letter $m$.
We write $L^{p}(\Omega; m)$ as $L^{p}(\Omega)$ simply.
For a function $u$ on $B$, we use the notation $\fint_{B} u \,dx := m(B)^{-1} \int_{B} u \, dx$.
For a ball $B = B(x, R) = \{ y \colon \dist(x, y) < R \}$ and $t > 0$,
we denote $B(x, t R)$ by $t B$.
The letter $C$ denotes various constants.

\section{Preliminaries}\label{sec:preliminaries}

We first clarify our terminology for measures.
We denote by $\mathcal{M}(\Omega)$
the set of all real-valued continuous linear functionals $\nu$ on $C_{c}(\Omega)$.
We have the identification
$\mathcal{M}(\Omega) = \left\{ \nu_{+} - \nu_{-} \colon \nu_{\pm} \in \mathcal{M}^{+}(\Omega) \right\}$
(see, \cite[Theorem 2, III.12]{MR2018901}).
If $\nu \in \mathcal{M}(\Omega)$, the total variation $|\nu| = \nu_{+} + \nu_{-}$ of $\nu$ belongs to $\mathcal{M}^{+}(\Omega)$.
For $\nu \in \mathcal{M}(\Omega)$ and $f \in L^{1}_{\loc}(\Omega; |\nu|)$,
we denote by $f \nu$ the linear functional 
\[
C_{c}(\Omega) \ni \varphi \mapsto \int_{\Omega} \varphi f \, d \nu.
\]

Next, we recall properties of Sobolev functions and variational capacities. 
We denote by $W^{1, p}(\Omega)$ the set of all weakly differentiable functions on $\Omega$
such that
\[
\| u \|_{W^{1, p}(\Omega)}
:=
\left(
\int_{\Omega} |u|^{p} \, dx + \int_{\Omega} |\nabla u|^{p} \, dx 
\right)^{1 / p}
\]
is finite.
The space $W_{0}^{1, p}(\Omega)$ is the closure of $C_{c}^{\infty}(\Omega)$ with respect to $\| \nabla \cdot \|_{L^{p}(\Omega)}$.
Let $U \subset \R^{n}$ be open.
For a compact subset $K \subset U$, the variational $p$-capacity $\capacity_{p}(K, U)$ of the condenser $(K, U)$ is defined by \eqref{eqn:variational_capacity}.
For $E \subset \Omega$, we define
\[
\capacity_{p}(E, \Omega)
:=
\inf_{\substack{ E \subset U \subset \Omega \\ U \colon \text{open} }}
\sup_{K \subset U \colon \text{compact}}
\capacity_{p}(K, \Omega).
\]
By \cite[Theorem 2.2]{MR2305115}, if $\{ E_{k} \}_{k = 1}^{\infty}$ is an increasing sequence of subsets of $\Omega$, then we have
\begin{equation}\label{eqn:cap_monotonicity}
\capacity_{p}\left( \bigcup_{k = 1}^{\infty} E_{k}, \Omega \right) = \lim_{k \to \infty} \capacity_{p}(E_{k}, \Omega).
\end{equation}
As in \cite[Theorem 2.5]{MR2305115}, Choquet's theorem
\begin{equation}\label{eqn:choquet}
\capacity_{p}\left( E, \Omega \right) = \sup_{K \subset E: \text{compact}} \capacity_{p}(K, \Omega)
\end{equation}
holds, where $E$ is any Borel subset of $\Omega$.
We say that a property holds $p$-\textit{quasieverywhere} (q.e.)
if it holds except on a set of $p$-capacity zero.
An extended real valued function $u$ on $\Omega$ is called as \textit{quasicontinuous}
if for every $\epsilon > 0$ there exists an open set $G$ such that
$C_{p}(G) < \epsilon$ and $u|_{\Omega \setminus G}$ is continuous.
Every $u \in W^{1, p}_{\loc}(\Omega)$ has a quasicontinuous representative $\tilde{u}$
such that $u = \tilde{u}$ a.e.

Next, we introduce Morrey spaces of finite signed measures.
Let by $\mathcal{M}_{b}(\Omega)$ denote the space of all regular finite signed measures.
Following the notation in \cite{MR3467116}, we define
\[
\| \mu \|_{\mathcal{M}^{ \lambda }(\Omega)}
:=
\sup_{ \substack{ x \in \Omega \\ 0 < r < \infty } } r^{- \lambda} |\mu|(\Omega \cap B(x, r))
\]
for $0 \le \lambda \le n$.
The \textit{Morrey space} $\mathcal{M}^{\lambda}(\Omega)$ consists of all
$\mu \in \mathcal{M}_{b}(\Omega)$ such that $\| \mu \|_{\mathcal{M}^{\lambda}(\Omega)} < \infty$.
We also define $\mathcal{M}^{\lambda}_{\loc}(\Omega)$ by
\[
\mathcal{M}^{\lambda}_{\loc}(\Omega)
:=
\left\{ \nu \in \mathcal{M}(\Omega) \colon \nu|_{\Omega'} \in \mathcal{M}^{\lambda}(\Omega') \ \text{for any} \ \Omega' \Subset \Omega \right\},
\]
where $\nu|_{\Omega'}$ is the restriction of $\nu$ to $\Omega'$.


\begin{example}\label{example:finite1}
Let $m$ be the Lebesgue measure, and let $f \in L^{q}(\Omega) = L^{q}(\Omega, m)$.
Then, $f m \in \mathcal{M}^{n - n / q}(\Omega)$ if $q < \infty$ and $f m \in \mathcal{M}^{n}(\Omega)$ if $q = \infty$.
In particular, the Lebesgue measure itself belongs to $\mathcal{M}^{n}(\Omega)$.
\end{example}


\begin{example}\label{example:finite3}
If $\sigma = \mathcal{H}^{n - 1} \lfloor_{E}$ is the surface measure of a smooth surface $E$,
then, $\sigma \in \mathcal{M}^{n - 1}(\Omega)$.
\end{example}


\begin{lemma}\label{lem:HW}
If $n - p < \lambda \le n$, then $\mathcal{M}^{\lambda}(\Omega) \subset (W_{0}^{1, p}(\Omega))^{*}$.
\end{lemma}

\begin{proof}
This follows from the Hedberg-Wolff theorem directly.
Let $\mu \in \mathcal{M}^{\lambda, +}(\Omega)$.
By assumption, we have
\[
\begin{split}
\mathbf{W}_{1, p}^{2 \diam(\Omega)} \mu_{+} (x)
& :=
\int_{0}^{2\diam(\Omega)}
\left( \frac{\mu_{+}(B(x, r))}{r^{n - p}} \right)^{1 / (p - 1)} \, \frac{d r}{r}
\\
& \le
C \| \mu \|_{ \mathcal{M}^{ \lambda }(\Omega) }^{1 / (p - 1)}
\diam(\Omega)^{ (\lambda - n + p) / (p - 1) }.
\end{split}
\]
Using assumption on $\mu$ again, we get
\[
\int_{\Omega} \mathbf{W}_{1, p}^{2 \diam(\Omega)} \mu_{+} \, d \mu_{+} 
\le
C \| \mu \|_{ \mathcal{M}^{ \lambda }(\Omega) }^{p / (p - 1)} \diam(\Omega)^{\lambda + (\lambda - n + p) / (p - 1)}.
\]
It follows from \cite[Remark 21.19]{MR2305115} that $\mu_{+} \in (W_{0}^{1, p}(\Omega))^{*}$.
The same argument valid for $\mu_{-}$. 
\end{proof}

We denote by $\M^{+}_{0}(\Omega)$ the set of all $\nu \in \mathcal{M}^{+}(\Omega)$ which satisfy
\[
\capacity_{p}(E, \Omega) = 0 \ \implies \nu(E) = 0
\]
for any Borel set $E \subset \Omega$.
It is known that if $\nu \in (W_{0}^{1, p}(\Omega))^{*} \cap \mathcal{M}^{+}(\Omega)$, then $\nu \in \mathcal{M}_{0}^{+}(\Omega)$
(\cite[Proposition 21.8]{MR2305115}).
Therefore, for $n - p <  \lambda \le n$, we have $\mathcal{M}^{\lambda}_{\loc}(\Omega) \subset \mathcal{M}_{0}(\Omega)$, where
\[
\mathcal{M}_{0}(\Omega) := \left\{ \mu \in \mathcal{M}(\Omega) \colon |\mu| \in \mathcal{M}_{0}^{+}(\Omega) \right\}.
\]

Finally, we summarize the elliptic regularity estimates to be used later.
Elliptic partial differential equations with $\mathcal{M}^{\lambda}(\Omega)$ coefficients have been studied by many authors.
The following estimates are consequence of \cite[Corollary 1.95]{MR1461542}.
Their proof is same as \cite[Theorems 3.12 and 3.13]{MR1461542}, except for replacing functions by measures.

\begin{lemma}\label{lem:global_boundedness}
Suppose that $\nu \in \mathcal{M}^{\lambda}(\Omega)$ for some $n - p <  \lambda \le n$. 
Let $u \in W^{1, p}(\Omega)$ be a weak subsolution to
$- \divergence \mathcal{A}(x, \nabla u) = \nu$ in $\Omega$.
Then,
\[
\esssup_{\Omega} u
\le
\sup_{\partial \Omega} u
+
C_{1} \| \nu \|_{\mathcal{M}^{ \lambda }(\Omega)}^{1 / (p - 1)} \diam(\Omega)^{ (\lambda - n + p) / (p - 1) },
\]
where $\sup_{\partial \Omega} u := \inf \{ k \in \R^{n} \colon (u - k)_{+} \in W_{0}^{1, p}(\Omega) \}$,
$C_{1}$ is a constant depending only on $n$, $p$, $L$ and $q$.
\end{lemma}

\begin{lemma}\label{lem:weak_harnack}
Suppose that $\nu \in \mathcal{M}^{\lambda}(2B)$ for some $n - p < \lambda \le n$.
Let $u \in W^{1, p}(2B)$ be a nonnegative weak supersolution to $- \divergence \mathcal{A}(x, \nabla u) = \nu$ in $2B$.
Then, for each $0 < s < n (p - 1) / (n - p)$,
there exists a constant $C$ depending only on $n$, $p$, $L$, $\lambda$ and $s$ such that
\begin{equation}\label{eqn:WHI}
\left( \fint_{B} u^{s} \, dx \right)^{1 / s}
\le
C \left( \essinf_{B} u + \| \nu \|_{\mathcal{M}^{ \lambda }(2B)}^{1 / (p - 1)} \diam(B)^{ (\lambda - n + p) / (p - 1) } \right).
\end{equation}
\end{lemma}

The following H\"{o}lder estimate is a standard consequence of Lemma \ref{lem:weak_harnack}.
We refer to \cite[Section 8.9]{MR1814364} for detail. 
Hence, we may replace $\esssup$ by $\sup$ without loss of generality.

\begin{lemma}\label{lem:interior_hoelder_esti}
Suppose that $\nu \in \mathcal{M}^{\lambda}(2B)$ for some $n - p < \lambda \le n$.
Let $u \in W^{1, p}(2B)$ be a nonnegative weak supersolution to $- \divergence \mathcal{A}(x, \nabla u) = \nu$ in $2B$.
Then, for any $\theta \in (0, 1)$, we have
\[
\osc_{\theta B} u
\le
C \theta^{\alpha_{0}}
\left(
\osc_{B} u
+
\| \nu \|_{\mathcal{M}^{ \lambda }(2B)}^{1 / (p - 1)} \diam(B)^{ (\lambda - n + p) / (p - 1) }
\right),
\]
where $\osc u := \sup u - \inf u$ and $C$ and $\alpha_{0}$ are constants depending only on $n$, $p$, $L$ and $\lambda$.
\end{lemma}

Below, we assume the following capacity density condition:
\begin{equation}\label{eqn:CDC}
\exists \gamma > 0 \quad \text{s.t.} \quad
\frac{ \capacity_{p}( \cl{B(\xi, R)} \setminus \Omega, B(\xi, 2R)) }{ \capacity_{p}( \cl{B(\xi, R)}, B(\xi, 2R)) } \ge \gamma \quad \forall R > 0, \ \forall \xi \in \Gamma,
\end{equation}
where $\Gamma$ is a closed subset of $\partial \Omega$.

The following boundary estimate follows from Lemma \ref{lem:weak_harnack} and arguments in \cite[Section 8.10]{MR1814364}. 
See also \cite{MR492836} and \cite[Chapter 4]{MR1461542}.

\begin{lemma}\label{lem:boundary_hoelder_esti}
Assume that \eqref{eqn:CDC} holds with $\Gamma = \{ \xi \}$.
Suppose that $\nu \in \mathcal{M}^{\lambda}(\Omega)$ for some $n - p < \lambda \le n$. 
Let $B$ be a ball centered at $\xi \in \Gamma$ with radius $R$. 
Let $u$ be a weak solution to $- \divergence \mathcal{A}(x, \nabla u) = \nu$ in $\Omega \cap B(\xi, R)$.
Then, for any $\theta \in (0, 1)$, we have
\[
\osc_{\Omega \cap \theta B} u
\le
C_{2} \theta^{\alpha} \osc_{\Omega \cap B} u
+
\osc_{\partial \Omega \cap B} u
+
C_{3} \| \nu \|_{\mathcal{M}^{ \lambda }(\Omega)}^{1 / (p - 1)} R^{ (\lambda - n + p) / (p - 1) },
\]
where $C_{2}$, $C_{3}$ and $\alpha$ are constants depending only on $n$, $p$, $L$, $\lambda$ and $\gamma$.
\end{lemma}

\section{Construction of supersolutions}\label{sec:barrier}

The main result of this section is Theorem \ref{thm:barrier}.
To state it, we introduce the following new Morrey spaces.

\begin{definition}\label{def:ungrounded_morrey}
Let $\Gamma \subset \partial \Omega$ be closed.
For $0 \le \lambda \le n$, we define
\[
\mathsf{M}_{\Gamma}^{\lambda}(\Omega)
:=
\left\{
\nu = \nu_{+} - \nu_{-}  \colon \nu_{\pm} \in \mathcal{M}^{+}(\Omega), \  \| \nu \|_{\mathsf{M}_{\Gamma}^{\lambda}(\Omega)} < \infty
\right\},
\]
where 
\begin{equation}\label{eqn:def_of_norm}
\| \nu \|_{\mathsf{M}_{\Gamma}^{\lambda}(\Omega)}
:=
\sup_{ \substack{x \in \Omega \\ 0 < r < \delta_{\Gamma}(x) / 2 } }
r^{- \lambda} |\nu|(\Omega \cap B(x, r)).
\end{equation}
We denote by $\mathsf{M}_{\Gamma}^{\lambda, +}(\Omega)$ the set $\mathsf{M}_{\Gamma}^{\lambda}(\Omega) \cap \mathcal{M}^{+}(\Omega)$.
\end{definition}

For $\emptyset \neq \Gamma \subsetneq \Gamma' \subsetneq \partial \Omega$, we have
\begin{equation}\label{eqn:inclusion}
\mathcal{M}^{\lambda}(\Omega)
=
\mathsf{M}^{\lambda}_{\emptyset}(\Omega)
\subset
\mathsf{M}^{\lambda}_{\Gamma}(\Omega)
\subset
\mathsf{M}^{\lambda}_{\Gamma'}(\Omega)
\subset
\mathsf{M}^{\lambda}_{\partial \Omega}(\Omega)
\subset
\mathcal{M}^{\lambda}_{\loc}(\Omega).
\end{equation}
Generally, these inclusions hold strictly (see, Remark \ref{rem:inclusion}).

\begin{theorem}\label{thm:barrier}
Assume that a bounded open set $\Omega$ satisfies \eqref{eqn:CDC}.
Suppose that $\nu \in \mathsf{M}_{\Gamma}^{\lambda, +}(\Omega)$ for some $n - p < \lambda \le n$.
Then, there exists $s \in W^{1, p}_{\loc}(\Omega) \cap C(\Omega)$ satisfying
\begin{equation}\label{eqn:barrier}
- \divergence \mathcal{A}\left(x, \nabla s \right) \ge \nu \quad \text{in} \ \Omega
\end{equation}
and
\begin{equation}\label{eqn:bound_of_barrier}
\begin{split}
& 
\frac{1}{C} \, (\sup_{\Omega} \delta_{\Gamma}( \cdot ) )^{\beta - \beta_{0}} \| \nu \|_{\mathsf{M}_{\Gamma}^{ \lambda }(\Omega)}^{1 / (p - 1)} \delta_{\Gamma}(x)^{\beta_{0}}
\\
& \quad
\le
s(x)
\le
C \, ( \sup_{\Omega} \delta_{\Gamma}( \cdot ) )^{\beta - \beta_{0}} \| \nu \|_{\mathsf{M}_{\Gamma}^{ \lambda }(\Omega)}^{1 / (p - 1)} \delta_{\Gamma}(x)^{\beta_{0}}
\end{split}
\end{equation}
for all $x \in \Omega$,
where $C$ and $\beta_{0}$ are positive constants depending only on $n$, $p$, $L$, $\lambda$ and $\gamma$,
and $\beta = (\lambda - n + p) / (p - 1)$.
\end{theorem}

We prove this theorem using the following two lemmas.

\begin{lemma}[{\cite[Lemma 3.1]{hara2022strong}}]\label{lem:barayage}\label{lem:glueing}
Let $f_{i} \in L^{1}_{\loc}(\Omega, \nu)$ for $i = 1, 2$.
Let $u_{1}, u_{2} \in W^{1, p}_{\loc}(\Omega)$ be weak supersolutions to
$- \divergence \mathcal{A}(x, \nabla u_{i}) = f_{i} \nu$ in $\Omega$.
Then,
\begin{equation*}\label{eqn:glueing}
- \divergence \mathcal{A}(x, \nabla \min \{ u_{1}, u_{2} \} ) 
\ge
\min \{ f_{1}, f_{2} \} \nu \quad \text{in} \ \Omega
\end{equation*}
in the sense of distributions.
\end{lemma}

\begin{lemma}\label{lem:auxiliary_func}
Assume that \eqref{eqn:CDC} holds.
Let $B$ be a ball centered at $\xi \in \Gamma$ with radius $R \le 1$. 
Let $\nu \in \mathcal{M}^{\lambda, +}(\Omega)$.
Let $\beta$, $C_{1}$, $C_{2}$ $C_{3}$ and $\alpha$ be constants in Lemmas \ref{lem:global_boundedness} and \ref{lem:boundary_hoelder_esti}.
Set $\beta_{0} := \min\{ \alpha / 2, \beta \}$ and take a constant $\theta$ such that
\begin{equation}\label{eqn:assumption_on_theta}
0 < 
\theta 
\le
\left( \frac{1}{4^{\alpha} \cdot 16 \cdot C_{2}} \right)^{1 / (\alpha - \beta_{0})},
\end{equation}
Assume also that
\begin{equation}\label{eqn:smallness_of_nu}
\| \nu \|_{\mathcal{M}^{\lambda}(B)}^{1 / (p - 1)}
\le \min \left\{ \frac{\theta^{\beta_{0}}}{8 C_{3}}, \frac{1}{C_{1}} \right\}.
\end{equation}
Let $u \in W^{1, p}(\Omega) \cap C(\Omega)$ be a solution to the problem
\begin{equation}\label{eqn:def_of_auxiliary_func}
\begin{cases}
\displaystyle
- \divergence \mathcal{A}(x, \nabla u) = \nu & \text{in} \ \Omega \cap B,
\\
u \in \eta + W_{0}^{1, p}(\Omega \cap B),
\end{cases}
\end{equation}
where $\eta \in C^{\infty}(\cl{\Omega})$ such that
\[
\eta = \frac{1}{4} (\theta R)^{\beta_{0}} \ \text{on} \ \cl{\Omega \cap B / 2},
\quad
\eta = R^{\beta_{0}} \ \text{on} \ \cl{\Omega} \setminus B,
\]
\[
\frac{1}{4} (\theta R)^{\beta_{0}} \le \eta \le R^{\beta_{0}} \ \text{in} \ \cl{\Omega}.
\]
Then, we have
\begin{align}
\frac{1}{4} (\theta R)^{\beta_{0}} \le u \le 2 R^{\beta_{0}} \quad & \text{in} \ \Omega \cap B, \label{eqn:global_bound_of_auxiliary_func}
\\
u \le \frac{1}{2} (\theta R)^{\beta_{0}} \quad & \text{on} \ \Omega \cap 2 \theta B. \label{eqn:boundary_bound_of_auxiliary_func}
\end{align}
\end{lemma}

\begin{proof}
We first confirm the existence of $u$.
By Lemma \ref{lem:HW}, we have $\nu \in (W_{0}^{1, p}(\Omega \cap B))^{*}$.
Therefore, by the Minty-Browder theorem, there exists a solution to \eqref{eqn:def_of_auxiliary_func} (see, \cite[p.177]{MR0259693} for detail).
Continuity of $u$ in $\Omega$ follows from Lemmas \ref{lem:interior_hoelder_esti} and \ref{lem:boundary_hoelder_esti}.
The latter inequality in \eqref{eqn:global_bound_of_auxiliary_func} is a consequence of
Lemma \ref{lem:global_boundedness}, \eqref{eqn:smallness_of_nu} and that $R^{\beta} \le R^{\beta_{0}}$.
The former inequality in \eqref{eqn:global_bound_of_auxiliary_func} follows from the comparison principle.
Let us prove \eqref{eqn:boundary_bound_of_auxiliary_func}. By assumption, we have $C_{2} 4^{\alpha} \theta^{\alpha - \beta_{0}} \le 1 / 16$.
Therefore, Lemma \ref{lem:boundary_hoelder_esti} gives
\[
\begin{split}
\sup_{\Omega \cap 2 \theta B} u
& \le
C_{2} 4^{\alpha} \theta^{\alpha - \beta_{0}} \theta^{\beta_{0}} \sup_{\Omega \cap B / 2} u
+
\sup_{\partial \Omega \cap B / 2} u
+
C_{3} \| \nu \|_{\mathcal{M}^{ \lambda }(\Omega \cap B / 2)}^{1 / (p - 1)} R^{\beta}
\\
& \le
\frac{1}{16} \theta^{\beta_{0}} \cdot 2 R^{\beta_{0}}
+
\sup_{\partial \Omega \cap B / 2} u
+
C_{3}\| \nu \|_{\mathcal{M}^{ \lambda }(\Omega \cap B / 2)}^{1 / (p - 1)} R^{\beta}
\\
& \le
\frac{1}{8} \theta^{\beta_{0}} R^{\beta_{0}}
+
\frac{1}{4} (\theta R)^{\beta_{0}}
+
\frac{1}{8} \theta^{\beta_{0}} R^{\beta_{0}}
=
\frac{1}{2} (\theta R)^{\beta_{0}}.
\end{split}
\]
Here, we used \eqref{eqn:global_bound_of_auxiliary_func} and that $R^{\beta} \le R^{\beta_{0}}$.
\end{proof}

\begin{proof}[Proof of Theorem \ref{thm:barrier}]
By rescaling $\Omega$, we may assume that $\sup_{x \in \Omega} \delta_{\Gamma}(x) \le 1$.
We also assume that $\| \nu \|_{\mathsf{M}_{\Gamma}^{ \lambda }(\Omega)}$ is smaller than a constant to be determined later.

Let $\beta_{0}$ be a constant in Lemma \ref{lem:auxiliary_func}, and take $0 <  \theta < 8^{-1 / \beta_{0}}$ satisfying \eqref{eqn:assumption_on_theta}.
For each $k \in \Z$, we set $\Gamma_{k} := \{ x \in \Omega \colon \delta_{\Gamma} \le \theta^{k} \}$.
Consider $\nu_{k} := \mathbf{1}_{\Omega \setminus \Gamma_{k + 2}} \nu$.
If we take a ball $B = B(\xi, \theta^{k})$ such that $\xi \in \Gamma$, then we have
\[
\| \nu_{k} \|_{ \mathcal{M}^{ \lambda }(\Omega \cap B) }
=
\sup_{ \substack{x \in \Omega \cap B \\ 0 < r < 2 \theta^{k}} }
r^{-\lambda}
\nu_{k}( (\Omega \cap B) \cap B(x, r) ).
\]
We cover $B(x, r)$ by finite balls $\{ \hat{B}_{i} \}_{i = 1}^{\mathcal{N}}$ with radius $\theta^{2} r / 6$,
where $\mathcal{N}$ is a constant depending only on $n$ and $\theta$.
If $\hat{B}_{i} \cap (\Omega \setminus \Gamma_{k + 2}) \neq \emptyset$, then $2 \hat{B}_{i} \cap \Gamma = \emptyset$.
Therefore, we have
\[
\begin{split}
\nu_{k}( (\Omega \cap B) \cap B(x, r) )
& \le
\sum_{i = 1}^{ \mathcal{N} } \nu_{k}( (\Omega \cap B) \cap \hat{B}_{i} )
\\
& \le
\sum_{i = 1}^{ \mathcal{N} } \nu_{k}( \hat{B}_{i} )
\le
\mathcal{N} \| \nu \|_{\mathsf{M}_{\Gamma}^{\lambda}(\Omega)} r^{\lambda}.
\end{split}
\]
Combining the two inequalities, we get
\[
\| \nu_{k} \|_{ \mathcal{M}^{\lambda}(\Omega \cap B) }
\le
\mathcal{N} \| \nu \|_{\mathsf{M}_{\Gamma}^{\lambda}(\Omega)}.
\]
Below, we assume that the left-hand side satisfies \eqref{eqn:smallness_of_nu}.

Next, we cover $\Gamma_{k + 1}$ by finitely many balls $\{ 2 \theta B_{k, j} \}_{j \in J_{k}}$,
where $B_{k, j}$ is a ball centered at $\xi_{j} \in \Gamma$ with radius $\theta^{k}$.
Note that
\[
\Gamma_{k + 1} \subset D_{k} := \Omega \cap \bigcup_{B_{j} \in J_{k}} B_{j} \subset \Gamma_{k}.
\]
For each $k \in \Z$ and $j \in J_{k}$, let $u_{k, j}$ be a function in Lemma \ref{lem:auxiliary_func} ($\Omega = \Omega \cap B_{k, j}$, $B = B_{k, j}$, $\nu = \nu_{k}$).
Using them, we define $v_{k} \in W^{1, p}_{\loc}(\Omega) \cap C(\Omega)$ by
\[
v_{k}(x)
:=
\begin{cases}
\displaystyle
\min_{ \substack{ j \in J_{k} \\ B_{k, j} \ni x} } u_{k, j}(x) &  x \in D_{k},
\\
\theta^{k \beta_{0}} & x \in \Omega \setminus D_{k}.
\end{cases}
\]
By \eqref{eqn:global_bound_of_auxiliary_func} and \eqref{eqn:boundary_bound_of_auxiliary_func}, we have
\begin{align}
\frac{1}{4} \theta^{(k + 1) \beta_{0}} \le v_{k} \le 2 \theta^{k \beta_{0}} \quad \text{in} \ \Gamma_{k}, \label{eqn:patched_func_01}
\\
v_{k} \le \frac{1}{2} \theta^{(k + 1) \beta_{0}} \quad \text{on} \ \Gamma_{k + 1}.  \label{eqn:patched_func_03}
\end{align}
Also, by Lemma \ref{lem:glueing}, we have
\begin{equation}\label{eqn:rigidity_eqn}
- \divergence \mathcal{A}(x, \nabla v_{k}) \ge \nu_{k} \quad \text{in} \ D_{k}.
\end{equation}
Define a function $s \in W^{1, p}_{\loc}(\Omega) \cap C(\Omega)$ by
\[
s(x) = \inf_{\Gamma_{k} \ni x} v_{k}(x).
\]
Recall that $\theta^{- \beta_{0}} > 8$.
Hence, for $k - 2 \ge k'$, we have
\begin{equation*}\label{eqn:bound_of_barrier2}
v_{k}(x)
\le
2 \theta^{k \beta_{0}}
=
\left( 2 \theta^{\beta_{0}} \right) \theta^{(k - 1)\beta_{0}}
<
\frac{1}{4} \theta^{(k' + 1) \beta_{0}}
\le
v_{ k' }(x),
\quad \forall x \in \Gamma_{k}.
\end{equation*}
Therefore,
\[
s(x)
=
\min\{ v_{k - 1}(x), v_{k}(x) \}
\quad \text{in} \ D_{k} \setminus \Gamma_{k + 1}.
\]
Meanwhile, since
\[
v_{k - 1}
\le
\frac{1}{2} \theta^{k \beta_{0}}
<
\theta^{k \beta_{0}}
=
v_{k},
\quad \text{on} \ \Gamma_{k} \setminus D_{k},
\]
there is an open neighborhood $O_{k}$ of $\Gamma_{k} \setminus D_{k}$ such that
\[
s = v_{k - 1}
\quad \text{in} \ O_{k}.
\]
By Lemma \ref{lem:glueing}, $- \divergence \mathcal{A}(x, \nabla s) \ge \nu_{k} = \nu$
in an open neighborhood of $\Gamma_{k} \setminus \Gamma_{k + 1}$.
Therefore, $s$ satisfies the desired inequality \eqref{eqn:barrier}.
The pointwise two-sided estimate \eqref{eqn:bound_of_barrier} follows from \eqref{eqn:patched_func_01}.
\end{proof}

Theorem \ref{thm:barrier} is a very strong statement.
In fact, \eqref{eqn:CDC} follows from the following weaker statement.

\begin{theorem}\label{thm:necessity}
Let $\Omega$ be an open set, and let $\Gamma$ be a closed subset of $\partial \Omega$.
Assume that there are constants $C > 0$ and $0 < \beta_{0} \le 1$, such that,
for any $\xi \in \Gamma$, $0 < R < \diam(\Omega)$,
there exists a nonnegative weak supersolution $s \in W^{1, p}_{\loc}(\Omega \cap B(\xi, R))$ to
$- \laplacian_{p} s = 0$ in $\Omega \cap B(\xi, R)$
satisfying
\[
\begin{split}
\frac{1}{C} |x - \xi|^{\beta_{0}}
\le
s(x)
\le
C  |x - \xi|^{\beta_{0}}
\end{split}
\]
for all $x \in \Omega \cap B(\xi, R)$.
Then, there exists a positive constant $\gamma$ depending only on $n$, $p$, $C$ and $\beta_{0}$
such that \eqref{eqn:CDC} holds.
\end{theorem}

\begin{proof}
Set $B = B(\xi, R)$.
Let $h$ be the minimizer of the problem
\[
\inf_{ \substack{ h \in W_{0}^{1, p}(2B) \\ h \ge 1 on \cl{B} \setminus \Omega} }
\int_{\Omega} |\nabla h|^{p} \, dx.
\]
Then, $1 - h$ satisfies $- \laplacian_{p} (1 - h) = 0$ in $\Omega \cap 2B$.
Moreover, this function satisfies the boundary conditions
$1 - h \le 1$ on $\Omega \cap \partial B$ and $1 - h = 0$ on $\partial \Omega \cap B$.
Take a sequence of subdomains $\{ \Omega_{k} \}_{k = 1}^{\infty}$ of $ \Omega$ such that $\bigcup_{k = 1}^{\infty} \Omega_{k} = \Omega$,
and let $u_{k} \in W^{1, p}(\Omega_{k} \cap B)$ be the solutions to the problem
$- \laplacian_{p} u_{k} = 0$ in $\Omega_{k} \cap B$, $u_{k} = 1 - h$ on $\Omega_{k} \cap \partial B$ and $u_{k} = 0$ on $\partial \Omega_{k} \cap B$.
By the comparison principle for weak solutions,
\[
u_{k}(x) \le \frac{C}{R^{\beta_{0}}} s(x)
\]
for a.e. $x \in \Omega_{k} \cap B$, where $s$ is a function in the statement.
Taking the limit $k \to \infty$, we find that
\[
1 - h(x) \le \frac{C}{R^{\beta_{0}}} s(x)
\]
for a.e. $x \in \Omega \cap B$.
Hence, we have $1 - h(x) \le 1 / 2$ for all $x \in \Omega \cap \cl{c B}$, where $c := (2 C)^{-1 / \beta_{0}}$.
By combining this inequality with an explicit formula of $\capacity_{p}(B(x, r), B(x, R))$ (see, \cite[Example 2.12]{MR2305115}), 
we obtain \eqref{eqn:CDC}.
\end{proof}

\section{Existence of globally H\"{o}lder continuous solutions}\label{sec:existence}

The main result of this section is Theorem \ref{thm:main}. It is a direct consequence of Theorem \ref{thm:existence} and known interior regularity estimates.
Theorem \ref{thm:existence} itself is derived from Theorem \ref{thm:barrier} and the comparison principle.

\begin{definition}
Let $C_{c}^{\infty}( \cl{\Omega} \setminus \Gamma )$ be the set of all smooth functions vanishing on a neighborhood of $\Gamma \subset \partial \Omega$.
We define
\[
W^{1, p}_{\complement \Gamma}(\Omega)
:=
\left\{ u \in W^{1, p}_{\loc}(\Omega) \colon u \varphi \in W_{0}^{1, p}(\Omega), \ \forall \varphi \in C_{c}^{\infty}( \cl{\Omega} \setminus \Gamma ) \right\}.
\]
Additionally, for $0 < \beta \le 1$, we introduce a subspace $L^{\infty}_{\Gamma, \beta}(\Omega) \subset L^{\infty}(\Omega)$ by
\[
L^{\infty}_{\Gamma, \beta}(\Omega)
:=
\left\{ u \in L^{\infty}(\Omega) \colon \left\| u / \delta_{\Gamma}^{\beta} \right\|_{L^{\infty}(\Omega)} < \infty \right\}.
\]
\end{definition}

\begin{theorem}\label{thm:existence}
Let $\Omega$ be a bounded open set satisfying \eqref{eqn:CDC}. 
Suppose that $\nu \in \mathsf{M}_{\Gamma}^{\lambda}(\Omega)$ for some $n - p < \lambda \le n$. 
Then, there exists a weak solution
$u \in W^{1, p}_{\complement \Gamma}(\Omega) \cap C(\Omega)$ to \eqref{eqn:DE} satisfying
\begin{equation}\label{eqn:bound_of_u}
\left\| u / \delta_{\Gamma}^{\beta_{0}} \right\|_{L^{\infty}(\Omega)}
\le
C \, (\sup_{x \in \Omega} \delta_{\Gamma}(x) )^{\beta - \beta_{0}} \| \nu \|_{\mathsf{M}_{\Gamma}^{ \lambda }(\Omega)}^{1 / (p - 1)},
\end{equation}
where $C$ and $\beta_{0}$ are constants depending only on $n$, $p$, $L$, $\gamma$ and $\lambda$,
and $\beta = (\lambda - n + p) / (p - 1)$.
Moreover, this solution is unique in $W^{1, p}_{\complement \Gamma}(\Omega) \cap L^{\infty}_{\Gamma, \beta_{0}}(\Omega)$.
\end{theorem}

\begin{proof}
Set $D_{k} = \{ x \in \Omega \colon \delta_{\Gamma}(x) > 2^{-k} \}$.
Consider the Dirichlet problems
\[
\begin{cases}
- \divergence \mathcal{A}(x, \nabla v_{k}) = \mathbf{1}_{D_{k}} |\nu| & \text{in} \ \Omega, \\
v_{k} = 0 & \text{on} \ \partial \Omega.
\end{cases}
\]
By assumption, the right-hand side belongs to $\mathcal{M}^{\lambda}(\Omega)$, and thus, it also belongs to $(W_{0}^{1, p}(\Omega))^{*}$.
Therefore, this problem has a weak solution $v_{k} \in W_{0}^{1, p}(\Omega)$.
By Lemma \ref{lem:interior_hoelder_esti}, we have $v_{k} \in W_{0}^{1, p}(\Omega) \cap C(\Omega)$.
Using Theorem \ref{thm:barrier}, we take a nonnegative function $s \in W^{1, p}_{\loc}(\Omega) \cap C(\Omega)$ satisfying
\[
- \divergence \mathcal{A}(x, \nabla v_{k}) \ge |\nu| \quad \text{in} \ \Omega
\]
and \eqref{eqn:bound_of_barrier}.
By the comparison principle in \cite[Theorem 3.5]{MR4493271},
\begin{equation}\label{eqn:comparison_with_v}
0 \le v_{k}(x)  \le s(x)
\end{equation}
for all $x \in \Omega$.
Let $\{ u_{k} \}_{k = 1}^{\infty} \subset W_{0}^{1, p}(\Omega) \cap C(\Omega)$ be weak solutions to
\[
\begin{cases}
- \divergence \mathcal{A}(x, \nabla u_{k}) = \mathbf{1}_{D_{k}} \nu & \text{in} \ \Omega, \\
u_{k} = 0 & \text{on} \ \partial \Omega.
\end{cases}
\]
By Theorem \ref{thm:comparison},
\[
- v_{k}(x) \le u_{k}(x) \le v_{k}(x)
\]
for all $x \in \Omega$; therefore,
\begin{equation}\label{eqn:bound_by_s}
|u_{k}(x)| \le s(x)
\end{equation}
for all $x \in \Omega$.

Next, we consider the limit of $\{ u_{k} \}_{k = 1}^{\infty}$.
Take $\eta \in C_{c}^{\infty}( \cl{\Omega} \setminus \Gamma )$.
By using the test function $u \eta^{p}$, we get
\begin{equation}\label{eqn:bound_of_global_en}
\begin{split}
\int_{\Omega} |\nabla (u_{k} \eta)|^{p} \, dx
\le
C \left( \| \nabla \eta \|_{L^{\infty}(\Omega)}^{p} \int_{ \spt \eta } |u_{k}|^{p} \, dx + \int_{ \spt \eta } |u_{k}|  \, d \nu  \right).
\end{split}
\end{equation}
Fix $j \ge 1$, and consider $\eta = \eta_{j} \in C_{c}^{\infty}(\Omega)$ such that $\eta_{j} \equiv 1$ on $\cl{D_{j}}$.
The right-hand side of \eqref{eqn:bound_of_global_en} is estimated by assumption on $\nu$ and \eqref{eqn:bound_by_s}.
Thus, there exist a constant $C_{j}$ independent of $k$ such that
\[
\| \nabla u_{k} \|_{L^{p}(D_{j})} \le C_{j}.
\]
Meanwhile, by Lemma \ref{lem:interior_hoelder_esti}, we have
\[
\| u_{k} \|_{C^{\alpha_{0}}(D_{j})} \le C_{j}.
\]
Take a subsequence of $\{ u_{k} \}_{k = 1}^{\infty}$, $u \in W^{1, p}_{\loc}(\Omega) \cap C(\Omega)$, and $\xi \in L^{p'}_{\loc}(\Omega)$ such that
\begin{align}
u_{k} \to u \quad & \text{locally uniformly in $\Omega$,}
\label{eqn:weak_conv_u}
\\
\nabla u_{k} \wkto \nabla u \quad & \text{weakly in } L^{p}_{\loc}(\Omega),
\label{eqn:weak_conv_nabla}
\\
\mathcal{A}(x, \nabla u_{k}) \wkto \xi \quad & \text{weakly in } L^{p'}_{\loc}(\Omega).
\label{eqn:weak_conv_flux}
\end{align}
By \eqref{eqn:bound_by_s}, we have $u \in L^{\infty}_{\Gamma, \beta_{0}}(\Omega)$.
Meanwhile, it follows from \eqref{eqn:bound_by_s}, \eqref{eqn:bound_of_global_en} and \eqref{eqn:weak_conv_nabla} that $u \in W^{1, p}_{\complement \Gamma}(\Omega)$.
Thus, $u$ satisfies the desired Dirichlet condition.

It remains to show that $\xi = \mathcal{A}(x, \nabla u)$.
However, it follows from standard arguments for monotone operators (see, \cite[Theorem 3.75]{MR2305115} for instance).
Hence, this completes the proof.
\end{proof}

The uniqueness part of Theorem \ref{thm:existence} follows from the following comparison principle.
It is well-known if $\Gamma = \emptyset$.

\begin{theorem}\label{thm:comparison}
Let $u, v \in W^{1, p}_{\complement \Gamma}(\Omega) \cap L^{\infty}_{\Gamma, \beta}(\Omega)$
be weak sub- and supersolutions to \eqref{eqn:DE}, respectively.
Assume that $u$ and $v$ are $p$-quasicontinuous on $\Omega$.
Then, $u \le v$ q.e. on $\Omega$.
In particular, if $u$ and $v$ are weak solutions to \eqref{eqn:DE}, then $u = v$ q.e. on $\Omega$.
\end{theorem}

\begin{proof}
If the claim is false, then, the set $\{ x \in \Omega \colon u(x) > v (x) \}$ has positive capacity.
By \eqref{eqn:cap_monotonicity}, the set $E := \{ x \in \Omega \colon (u - v)(x) > 2 \epsilon \}$ has positive capacity for some $\epsilon > 0$.
Since $u, v \in L^{\infty}_{\Gamma, \beta}(\Omega)$, we have
\[
|u - v|(x) \le C \delta_{\Gamma}(x)^{\beta}
\]
for q.e. $x \in \Omega$, where $C := \| (u - v) / \delta_{\Gamma}^{\beta} \|_{L^{\infty}(\Omega)}$.
Consider the open set $D = \{ x \in \Omega \colon \delta_{\Gamma}(x)^{\beta} > \epsilon / C \}$.
By taking $\eta \in C^{\infty}(\R^{n})$ such that $\eta = 1$ on $\cl{D}$ and $\eta = 0$ near $\Gamma$, we find that $u, v \in W^{1, p}(D)$.
Since $(u - v - \epsilon)_{+} = 0$ q.e. on $\Omega \setminus D$,
it follows from \cite[Theorem 4.5]{MR2305115} that $(u - v - \epsilon)_{+} \in W_{0}^{1, p}(D)$.
Using density of $C_{c}^{\infty}(D)$ in $W_{0}^{1, p}(D)$, we get
\[
\int_{\Omega} \left( \mathcal{A}(x, \nabla u) - \mathcal{A}(x, \nabla u) \right) \cdot \nabla (u - v - \epsilon)_{+} \, dx \le 0.
\]
It follows from \eqref{eqn:monotonicity} that $\nabla u = \nabla v$ a.e. in $D$. 
Meanwhile, since $\capacity_{p}(E \setminus D, \Omega) = 0$, we have
\[
0
<
\capacity_{p}(E \cap D, \Omega).
\]
Let $K$ be any compact subset of $E \cap D$.
Since $u - v - \epsilon > \epsilon$ in $E \cap D$, we have
\[
\capacity_{p}(K, \Omega)
\le
\frac{1}{\epsilon^{p}} \int_{\Omega} |\nabla (u - v - \epsilon)_{+}|^{p} \, dx
=
\frac{1}{\epsilon^{p}} \int_{D} |\nabla (u - v - \epsilon)_{+}|^{p} \, dx.
\]
This contradicts to \eqref{eqn:choquet} because the right-hand side is zero.
\end{proof}


\begin{proof}[Proof of Theorem \ref{thm:main}]
Existence of a solution $u \in W^{1, p}_{\loc}(\Omega) \cap C(\Omega)$ to \eqref{eqn:DE} follows from Theorem \ref{thm:existence}.
Since $u(x)$ is controlled by $\delta(x)^{\beta_{0}}$, it also belongs to $C(\cl{\Omega})$.

Let us prove \eqref{eqn:hoelder_esti}. Set $\beta_{1} := \min\{ \beta_{0}, \alpha_{0} \}$,
where $\alpha_{0}$ and $\beta_{0}$ are constants in Lemma \ref{lem:interior_hoelder_esti} and Theorem \ref{thm:barrier}, respectively.
If $|x - y| > \max\{ \delta(x),  \delta(y) \} / 2$, then, \eqref{eqn:bound_of_u} gives
\[
\begin{split}
|u(x) - u(y)|
& \le
|u(x)| + |u(y)|
\\
& \le
\left( C  \| \nu \|_{\mathsf{M}_{\partial \Omega}^{ \lambda }(\Omega)}^{1 / (p - 1)} \right)
\inrad(\Omega)^{\beta - \beta_{0}} \left( \delta(x) + \delta(y) \right)^{\beta_{0}}
\\
& \le
\left( C  \| \nu \|_{\mathsf{M}_{\partial \Omega}^{ \lambda }(\Omega)}^{1 / (p - 1)} \right)
\inrad(\Omega)^{\beta - \beta_{1}} |x - y|^{\beta_{1}}.
\end{split}
\]
Therefore, we may assume that $|x - y| \le \delta(x) / 2$ without loss of generality.
Set $r := \delta(x)$. Since $B(x, 2r) \subset \Omega$,  Lemma \ref{lem:interior_hoelder_esti} yields
\[
\begin{split}
|u(x) - u(y)|
& \le
C \left( \frac{ |x - y| }{r} \right)^{\beta_{1}} \left( \osc_{B(x, r / 2)} u +  \| \nu \|_{\mathsf{M}^{q}(\Omega)}^{1 / (p - 1)} r^{\beta} \right) .
\end{split}
\]
Meanwhile, by \eqref{eqn:bound_of_u}, we have
\[
\begin{split}
|u(z)|
& \le
\left( C  \| \nu \|_{\mathsf{M}_{\partial \Omega}^{ \lambda }(\Omega)}^{1 / (p - 1)} \right)
\inrad(\Omega)^{\beta - \beta_{0}} r^{\beta_{0}}
\\
& \le
\left( C  \| \nu \|_{\mathsf{M}_{\partial \Omega}^{ \lambda }(\Omega)}^{1 / (p - 1)} \right)
\inrad(\Omega)^{\beta - \beta_{1}} r^{\beta_{1}}
\end{split}
\]
for all $z \in B(x, r / 2)$.
Combining the two inequalities, we obtain the desired estimate.
\end{proof}

As mentioned in the introduction, the condition \eqref{eqn:RZ} is necessary.
In particular, we get the following criterion.

\begin{corollary}\label{cor:criterion}
Let $\Omega$ be a bounded open set satisfying \eqref{eqn:CDC} with $\Gamma = \partial \Omega$, and let $\nu \in \mathcal{M}^{+}(\Omega)$.
Then, there exists a globally H\"{o}lder continuous weak solution $u \in W^{1, p}_{\loc}(\Omega) \cap C(\cl{\Omega})$ to \eqref{eqn:DE}
if and only if $\nu \in \mathsf{M}_{\partial \Omega}^{\lambda, +}(\Omega)$ for some $n - p <  \lambda \le n$.
\end{corollary}

\section{Applications to embedding theorems}\label{sec:embedding}

Using Theorem \ref{thm:main} and the Picone inequality (see, \cite{MR1618334}), we prove the following compact embedding result.

\begin{theorem}\label{thm:compact_embedding}
Let $\Omega$ be a bounded open set satisfying \eqref{eqn:CDC_intro}.
Suppose that $\nu \in \mathsf{M}_{\partial \Omega}^{\lambda, +}(\Omega)$ for some $n - p <  \lambda \le n$.
Then, the embedding $i \colon W_{0}^{1, p}(\Omega) \hookrightarrow L^{p}(\Omega; \nu)$ holds and is compact.
\end{theorem}

We first prove the boundedness.

\begin{theorem}\label{thm:embeding}
Let $\Omega$ be a bounded open set satisfying \eqref{eqn:CDC}.
Suppose that $\nu \in \mathsf{M}_{\Gamma}^{\lambda, +}(\Omega)$ for some $n - p <  \lambda \le n$.
Then, the embedding $i \colon W_{0}^{1, p}(\Omega) \to L^{p}(\Omega; \nu)$ holds.
Moreover, we have
\[
\sup_{ \| \nu \|_{\mathsf{M}_{\Gamma}^{ \lambda }(\Omega)} \le 1 }
\| \varphi \|_{L^{p}(\Omega; \nu)}^{p}
\le
C \, ( \sup_{x \in \Omega} \delta_{\Gamma}(x) )^{\lambda - n + p} \| \nabla \varphi \|_{L^{p}(\Omega)}^{p}
\]
for all $\varphi \in C_{c}^{\infty}(\Omega)$, where $C$ is a constant depending only on $n$, $p$, $\lambda$ and $\gamma$.
\end{theorem}

\begin{proof}
Using Theorem \ref{thm:existence}, we take a nonnegative bounded weak solution $u$ to $- \laplacian_{p} u = \nu$ in $\Omega$.
By \cite[Theorem 1]{MR1618334}, for differentiable functions $u$ and $\varphi$, we have
\begin{equation}\label{eqn:picone}
|\nabla u(x)|^{p - 2} \nabla u(x) \cdot \nabla (u(x)^{1 - p} \varphi(x)^{p}) \le |\nabla \varphi(x)|^{p}.
\end{equation}
Therefore,
\[
\int_{\Omega} \varphi^{p} d \nu
\le
\| u \|_{L^{\infty}(\Omega)}^{p - 1}
\int_{\Omega} \varphi^{p} u^{1-p} d \nu
\le
\| u \|_{L^{\infty}(\Omega)}^{p - 1}
\int_{\Omega} |\nabla \varphi|^{p} \, dx
\]
for any nonnegative $\varphi \in C_{c}^{\infty}(\Omega)$.
The right-hand side is estimated by \eqref{eqn:bound_of_u}.
\end{proof}

\begin{proof}[Proof of Theorem \ref{thm:compact_embedding}]
By Theorem \ref{thm:embeding}, $i$ is bounded. Hence, we consider compactness of it.

Define a sequence of operators $\{ i_{k} \}_{k = 1}^{\infty}$ by
\[
i_{k} \colon W_{0}^{1, p}(\Omega) \ni \varphi \mapsto \varphi \mathbf{1}_{ \Omega_{k} } \in L^{p}(\Omega; \nu),
\]
where $\Omega_{k} := \{ x \in \Omega \colon \delta(x) > 2^{-k} \}$.
Fix $k \ge 1$.
By assumption on $\nu$, we have
\[
\lim_{R \to 0} \sup_{ \substack{x \in \R^{n} \\ 0 < r < R} } r^{1 - n / p} \nu(\Omega_{k} \cap B(x, r))^{1 / p}
\le
\lim_{R \to 0} C R^{\beta (p - 1) / p}
=
0.
\]
Therefore, by \cite[Sect. 11.9.1, Theorem 3]{MR2777530}, $i_{k}$ is compact.

Using Theorem \ref{thm:main}, we take $u_{k} \in W^{1, p}_{\loc}(\Omega) \cap C(\cl{\Omega})$ such that
\[
\begin{cases}
- \laplacian_{p} u_{k} = \mathbf{1}_{ \Omega \setminus \Omega_{k} } \nu & \text{in} \ \Omega,
\\
u_{k} = 0 & \text{on} \ \partial \Omega.
\end{cases}
\]
By Theorem \ref{thm:comparison}, $\{ u_{k} \}_{k = 1}^{\infty}$ is nonincreasing with respect to $k$.
Consider $u = \lim_{k \to \infty} u$.
It follows from \cite[Theorem 3.78]{MR2305115} that $- \laplacian_{p} u = 0$ in $\Omega$.
Also, $u \in W^{1, p}_{\loc}(\Omega) \cap L^{\infty}_{\partial \Omega, \beta_{0}}(\Omega)$ clearly.
Therefore, $u = 0$ by Theorem \ref{thm:comparison}.
By Dini's theorem, $\{ u_{k} \}_{k = 1}^{\infty}$ converges to zero uniformly.
Meanwhile, by \eqref{eqn:picone}, we have
\[
\begin{split}
\| (i - i_{k}) \varphi \|_{L^{p}(\Omega, d \nu)}^{p}
=
\int_{\Omega} |\varphi|^{p} \mathbf{1}_{ \Omega \setminus \Omega_{k} } \, d \nu
\le
\| u_{k} \|_{L^{\infty}(\Omega)}^{p - 1}
\int_{\Omega} |\nabla \varphi|^{p} \, dx
\end{split}
\]
for any nonnegative $\varphi \in C_{c}^{\infty}(\Omega)$.
Since each $i_{k}$ is compact, $i$ is also compact.
\end{proof}

\section{Examples}\label{sec:examples}

In this section, we discuss examples of $\nu \in \mathsf{M}_{\partial \Omega}^{\lambda}(\Omega)$.

The following claim holds without any geometric assumption on $\Omega$.

\begin{proposition}\label{prop:dist}
Let $\Omega$ be an open set in $\R^{n}$.
Let $\mu \in \mathcal{M}^{\lambda}(\Omega)$ ($0 \le \lambda \le n$), and let $0 \le t \le \lambda$.
Then, we have
$\nu := \delta_{\Gamma}(x)^{- t} \, \mu \in \mathsf{M}_{\Gamma}^{ \lambda - t }(\Omega)$.
\end{proposition}

\begin{proof}
For any $x \in \Omega$ and $0 < r \le \delta_{\Gamma}(x) / 2$, we have
\begin{equation*}\label{eqn:negative_integrability}
\begin{split}
\int_{B(x, r)} \delta_{\Gamma}(x)^{- t} \, d |\mu|
& \le
\int_{B(x, r)} \left( \inf_{y \in B(x, r)} \delta_{\Gamma}(y) \right)^{-t} \, d |\mu|
\\
& \le
\| \mu \|_{ \mathsf{M}^{\lambda}(\Omega) } r^{\lambda - t}.
\end{split}
\end{equation*}
Thus, the claimed inclusion holds.
\end{proof}

\begin{example}\label{example:dist}
Let $m$ be the Lebesgue measure.
Then, $\nu := \delta^{-t} m \in \mathsf{M}_{\partial \Omega}^{n - t}(\Omega)$ for $0 \le t \le n$.
\end{example}

\begin{remark}\label{rem:inclusion}
Generally, $\nu$ in Example \ref{example:dist} is not finite on $\Omega$.
If $\Omega$ is a bounded Lipschitz domain,
then we can check that $\int_{\Omega} \delta^{-t} \, dx$ is finite if and only if $t < 1$ using graph representation of $\partial \Omega$.
Therefore, for $1 \le t < n$, we have
\[
\nu
\in
\mathsf{M}_{\partial \Omega}^{n - t}(\Omega) \setminus \mathcal{M}^{0}(\Omega)
\subset
\mathsf{M}_{\partial \Omega}^{n - t}(\Omega) \setminus \mathcal{M}^{n - t}(\Omega).
\]
By the same method, we can make examples showing that the inclusions in \eqref{eqn:inclusion} are strictly.

Further results of integrability of distance functions can be found in \cite{MR2125540,MR4306765} and the references therein.
They are depending on interior properties of $\Omega$.
\end{remark}

Applying Theorem \ref{thm:main} and Proposition \ref{prop:dist} to Example \ref{example:finite1}, we get the following.

\begin{corollary}\label{cor:existence_for_functions}
Suppose that $f \delta^{t} \in L^{q}(\Omega)$ for some $q \in (n / p, \infty]$ and $0 \le t < p - n / q$.
Then, $\nu = f m \in \mathsf{M}^{n - n / q  - t}(\Omega)$.
Assume further that \eqref{eqn:CDC} holds.
Then, there exists a weak solution $u$ to \eqref{eqn:DE} satisfying \eqref{eqn:hoelder_esti}.
\end{corollary}

Corollary \ref{cor:existence_for_functions} is a refinement of \cite[Theorem 6.3]{hara2022strong}.
It correspond to the results for ordinary differential equations under the assumption of smoothness or symmetry.
This can be considered that one-dimensional results hold if $\nu$ is leveled in the sense of $L^{\infty}(\Omega; m)$ functions.

On the other hand, if $\nu$ concentrates on a low-dimensional set, the range of $t$ is reduced.

\begin{corollary}\label{cor:surface}
Let $\sigma$ be a measure in Example \ref{example:finite3}.
Let $0 \le t < p - 1$, and assume that $|f(x)| \le C \delta(x)^{-t}$ for all $x \in \Omega$.
Then, $\nu := f \sigma \in \mathsf{M}^{n - 1 - t}(\Omega)$.
Assume further that \eqref{eqn:CDC} holds.
Then, there exists a weak solution $u$ to \eqref{eqn:DE} satisfying \eqref{eqn:hoelder_esti}.
\end{corollary}

\begin{remark}
In general, we cannot expect Corollary \ref{cor:surface} to hold for $t$ close to $p$.
Let us consider the case $\Omega$ is a polygon and $\mathcal{A}(x, z) = z$ ($p = 2$).
Then, the solution to \eqref{eqn:DE} is given by the Green potantial
\begin{equation}\label{eqn:green_potential}
u(x) = \int_{\Omega} G_{\Omega}(x, y) \, d \nu(y),
\end{equation}
where $G_{\Omega}(\cdot, \cdot)$ is the Green function of $\Omega$.
In this case, we can estimate the boundary behavior of $G(x, y)$ by using a Schwarz-Christoffel mapping and the boundary Harnack principle,
and it is not comparable to $\delta(y)$ near the corners.
More precisely, if the support of $\nu$ intersects a concave corner of $\partial \Omega$,
then \eqref{eqn:green_potential} is identically infinite for $t$ close to $2$.
Note that boundary Lipschitz estimate for the Green function does not follow from \eqref{eqn:CDC}
and that Lemma \ref{lem:boundary_hoelder_esti} only provides a small H\"{o}lder exponent.
\end{remark}

\begin{remark}
Even when $p=2$ in Corollary \ref{cor:surface}, $\nu$ is not a finite measure in general.
Counterexamples can be made with domains that have an outward cusp.
For example, fix $s \ge 1$ and consider
\[
\Omega = \{ (x, y) \in \R^{2} \colon 0 < x < 1,  |y| < x^{s} \}.
\]
Then,  $d \nu = \delta^{-t} d \mathcal{H}^{1} \lfloor _{ \{ 0 < x < 1, y = 0 \} }$ satisfies assumptions in Corollary \ref{cor:surface}.
However, $\nu$ is not finite if $s t > 1$.
\end{remark}

\section*{Acknowledgments}
This work was supported by JST CREST (doi:10.13039/501100003382) Grant Number JPMJCR18K3
and JSPS KAKENHI (doi:10.13039/501100001691) Grant Number 23H03798.
The author appreciates the helpful comments received during presentations at Meijo University and Nagoya University.
I would like to thank the referees and editors from other journals for their valuable time spent reviewing me. 
This version of the manuscript has been proofread by ChatGPT to improve the language quality.

\bibliographystyle{abbrv} 
\bibliography{reference}


\end{document}